\documentclass[a4paper,onecolumn,english]{svjour3}
\usepackage[T1]{fontenc}
\usepackage[utf8]{inputenc}
\usepackage{float}
\usepackage{textcomp}
\usepackage{amsmath}
\usepackage{amssymb}
\usepackage{graphicx}
\usepackage[authoryear]{natbib}

\makeatletter

\pdfpageheight\paperheight
\pdfpagewidth\paperwidth

\newcommand{\lyxmathsym}[1]{\ifmmode\begingroup\def\b@ld{bold}
  \text{\ifx\math@version\b@ld\bfseries\fi#1}\endgroup\else#1\fi}

\providecommand{\tabularnewline}{\\}
\floatstyle{ruled}
\newfloat{algorithm}{tbp}{loa}
\providecommand{\algorithmname}{Algorithm}
\floatname{algorithm}{\protect\algorithmname}

\smartqed  

\makeatother

\usepackage{babel}
\begin{document}

\title{Distributed Multi-objective Multidisciplinary Design Optimization
algorithms}

\author{Amir Noori}

\institute{Islamic Azad University, Karaj Branch, \\
Moazen BLVD, Rajaee-shahr, \\
P.O. Box 31485-313, Karaj, Iran\\
Email: amir.noori@kiau.ac.ir}
\maketitle
\begin{abstract}
This work proposes multi-agent systems setting for concurrent engineering
system design optimization and gradually paves the way towards examining
graph theoretic constructs in the context of multidisciplinary design
optimization problem. This paper adapts a cooridination strategy based
on the well-known nearest neighbor rule and corresponding distributed
constrained optimization method. The flow of the algorithm can be
described as follow; generated estimates of the optimal (shared design)
variables are exchanged locally with neighbor subspaces and then updated
by computing a weighted sum of the local and received estimates. To
comply with the consistency requirement, the resultant values are
projected to local constraint sets. By employing the existing rules
and results of the field, it has shown that the dual task of reaching
consensus and asymptotic convergence of the algorithms to locally
and globally optimal and consistent designs can be achieved. Finally,
simulations are provided to illustrate the effectiveness and capability
of the presented framework.

\keywords{Multidisciplinary design optimization, Consensus algorithms, Projection
methods, Distributed optimization}
\end{abstract}

\section{INTRODUCTION}

Development and optimization of complex engineering systems arise
from the challenges of effectively addressing the competing needs
of improving performance, reducing costs, and enhancing safety. Modern
complex engineering systems are usually heterogeneous, highly interconnected
and mutually interdependent, both physically and through a multitude
of information and communication networks \citep{W.M.Haddad2006}.
Examples of such systems include automobile and rail vehicles design,
naval architecture, electronics, computers, and micro-electro-mechanical
systems (MEMS), as well as system of systems such as air and ground
traffic networks, distributed manufacturing environments, and globally
distributed supply networks. 

Design and optimization of these complex engineered systems are multidisciplinary
in nature, far from optimal, and heavily constrained by both technical
and nontechnical considerations. It is therefore impossible for one
designer, or even a single design team, to consider the entire problem
as a single design problem. Therefore, the design process is decentralized
or distributed over a number of design teams that autonomously operates
on a single component or aspect of the system. This paper is about
Collaborative Multidisciplinary Design Optimization, or CMDO, and
deals with concurrent optimization of two or more coupled analysis
disciplines with distributed computation models and numerical optimization. 

Multidisciplinary design optimization \citep{Alexandrov1995} allows
designers to incorporate all relevant disciplines simultaneously.
The simultaneous optimization is superior to the commonly used sequential
design methods, since it can exploit the interactions between the
disciplines. Design is both analysis and synthesis, and is compromise
in the balance of conflicting requirements. However, finding the best
compromise by including all disciplines simultaneously significantly
increases the complexity of the problem. 

During the past three decades, decomposition-based design optimization
strategies \citep{Sobieski1997} as a natural approach have drawn
a great deal of attention of researchers to solve the design problem
of complex systems in a distributed way. However, most works in such
decentralized design optimization settings address a hierarchical
or sequential evaluation of a master problem and some disciplinary
sub-problems. The optimal set of design variables for the upper level
becomes the objective/constraints responses for the lower level. At
each level, disciplines are optimized separately. Accordingly, two
general methods are elaborated in literature; single-level methods,
which have centralized decision-making authority and do not allow
design decisions to be made at the disciplinary level, and multi-level
methods, which a central master optimization problem is introduced
to coordinate the interactions between the disciplinary sub-problems.
In some multi-level methods, such distributed authority for decision
making is sometimes referred to as distributed design optimization
schemes. However, in this paper, by distributed methods we mean the
same decentralized decision making capability without any master or
central coordinator. 

However, several recently proposed distributed methods are placed
in the multi-level category. Typically, these methods hierarchically
decompose the underlying design problem into sub-systems along the
lines of systems, subsystems, and components or usually partition
in a non-hierarchical fashion along disciplinary lines. A review of
single-level and multi-level methods can be found in \citep{Cramer1994}
and \citep{Tosserams2009}. 

Several partitioning and coordination methods have been proposed for
Multidisciplinary Design Optimization (MDO) problem including linear
decomposition method (OLD) \citep{Sobieski1982}, Concurrent Subspace
Optimization (CSSO) \citep{Sobieski1988}, BLISS \citep{Sobieski2003},
analytical target cascading (ATC) \citep{Michelena1999}, Quasiseparable
decomposition (QSD) approach \citep{Haftka2005}, penalty decomposition
formulation \citep{Demiguel2006}, Augmented Lagrangian Decomposition
method for Quasi-separable problems \citep{Tosserams2007}, and Collaborative
Optimization (CO) \citep{Braun1996}, Enhanced Collaborative Optimization
(ECO) \citep{Roth2008}. A full description of decomposition and coordination
strategies in multidisciplinary design optimization problems is beyond
the scope of this paper and was already reported in \citep{Tosserams2009}.
As mentioned before, most of these approaches are essentially decentralized
in the way that they organize a fully centralized problem in a hierarchical
(i.e., sequential) structure. It is clear that we have reached the
limits of what these approaches can do \citep{Allen2011}. To proceed,
we need a more rigorous and deeper understanding of complex engineered
systems and how they should be designed. 

The penalty relaxation methods \citep{Michelena1999}, \citep{Blouin2005},
\citep{Tosserams2008}) relax the coupling constraints of MDO problems
to arrive at subproblems with separable constraint sets. ALC \citep{Tosserams2008}
provides a flexible coordination structure, not necessarily hierarchical
that uses penalty relaxation methods in tandem with algorithms for
solving systems of equations. The convergence to KKT points of the
original problems is the main advantage of penalty methods. Some penalty
relaxations methods have only been developed for quasi-separable problems
coupled through a set of coupling variables; coupling objectives and
constraints are not allowed. In \citep{Tosserams2006}, an Augmented
Lagrangian Relaxation for Analytical Target Cascading using the Alternating
Directions Method of Multipliers are proposed. In \citep{Tosserams2008},
a new penalty relaxation coordination method is proposed that can
be used to solve MDO problems with coupling variables, a coupling
objective, and coupling constraints. However, ALC results in a very
large number of consistency constraints; only a subset is actually
required to ensure consistency. On the other hand, such methods rely
on excessively large penalty factors for sufficiently accurate solutions.
Several consistency constraints allocation guidelines have been proposed
for ALC implementations in \citep{Allison2010}. 

Our work is also related to game-based design approaches \citep{Hernandez2002}
\citep{Xiao2005} \citep{Ciucci2012}. The pioneer work of Lewis et
al \citep{Lewis1997} in the last nineties suggested a game theoretic
approach to model interactions in multidisciplinary design. In \citep{Hernandez2002},
the authors investigated a quadratic and eigen-based formulation to
enhance convergence speed. They conclude that passing more information
generally leads to convergence to a Pareto-optimal set. 

This work is built upon the constrained consensus method \citep{Nedic2009}
\citep{Nedic2010}, adapts a distributed computation framework for
general complex system design and optimization problem and gradually
paves the way towards examining graph theoretic constructs in the
context of multidisciplinary design optimization problem. Convergence
properties of the agreement protocol can be proved using existing
result on algebraic graph theory \citep{Godsil2001}, in particular,
spectral properties of the underlying graph. 

This paper considers the design problem of complex engineering systems,
consisting of several (including both local and global) objectives,
design variables and constraints corresponding to different disciplines
and assigned to several teams. In other words, system-wide disciplines
(i.e., those are spread throughout the problem space) are also allowed.
In particular, these distributed multi-objective multidisciplinary
design optimization problem is presented in a multi-agent setting,
providing necessary information passing structure to come up with
an appropriate decision. In brief, the flow of the proposed algorithm
can be described as follow; generated estimates of shared design variables
and optimal linking variables (if any) are exchanged directly among
subspaces. Afterwards, each agent computes a weighted sum of its estimates
and received estimates. Then, each agent projects the variables received
from other agents to its constraint sets to maintain consistency. 

Analysis shows that the performance of the algorithm largely depends
on the subproblems as well as the network structure and communication
protocols. Afterwards, our main focus is on developing distributed
algorithms that guarantees asymptotic consensus on the shared and
linking variables while maintains the feasibility with respect to
the constraint sets. The proposed multi-agent framework is adapted
to flexibly address both system-level (global) and discipline-level
(local) issues, without any requirement of objectives and constraints
relaxation. Moreover, it has been shown that the proposed framework
is in great accordance with the corresponding design and development
teams. Finally, the main feature that distinguishes the proposed method
from others is its structural flexibility, design autonomy, and rigor
mathematical and graphical representation. We believe that the most
challenging part of the method (especially for general design optimization
problem) is the implementation and analysis of its nonlinear part;
i.e., projection algorithm. In addition, coupling between subspaces
might be another challenging issue and require further investigations
both in theory and in experiment. Attention is focused in this paper
on the general case, and some companion papers will be devoted to
special cases that allow further analysis. 

The rest of the paper is organized as follows. Section II is devoted
to provide the necessary mathematical and graph theoretical foundations
and computational algorithms used for incremental coordination and
projection. Section III introduces the main problem we consider, discusses
the assumptions made in the proposed model, and presents our algorithm.
Section IV presents the performance analysis and the convergence issues
of the algorithm, followed by section V that demonstrates the algorithm
and compares the results with that of the well-known All-in-once (AIO)
method. Section VI concludes the paper.

\section{PRELIMINARIES}

In this section, we discuss the standard consensus algorithm and the
constrained consensus algorithm. Some of the equations and ideas for
this section covering consensus, projection and constrained consensus
algorithms, originate from \citep{Nedic2010}, demonstrating the existing
strong theoretical foundation. They are included here to ensure a
complete statement of the problem at hand and an accurate comparison
between coordination strategies.

\subsection{Consensus Algorithms}

Distributed average consensus algorithms are a class of iterative
update schemes that work based on the neighbor interactions. In recent
years, there has been a surge of interests in distributed computing
methods based on the average consensus algorithms \citep{Jadbabaie2003}\citep{Olfati-Saber2004},
and it has found applications in rendezvous, formation control, flocking,
attitude alignment, decentralized task assignment, and sensor networks.
Let’s consider a network of $i$ agents, represented by $V=\left\{ 1,2,\ldots,n\right\} $.
The neighbors of node $i$ is a set of nodes $j\in V$ where communicating
with node $j$ through a directed link $e=(i,j)$. At each time $k+1$,
we assume that agent $i$ receives information $z_{j}(k)$ from neighboring
agents $j$ and updates its estimate by adding a weighted sum of the
local discrepancies, i.e., the differences between neighboring node
values and its own. In \citep{Olfati-Saber2004}, Olfati-saber and
Murray show that the following linear dynamic system:

\begin{eqnarray}
z_{i}(k+1) & = & z_{i}(k)+\underset{j\in N_{i}}{\sum}a_{ij}(k)(z_{j}(k)-z_{i}(k))\label{eq:avg_con}
\end{eqnarray}

where $a_{ij}(k)$ is a weight associated with the edge $(i,j),j\in N_{i}$
and $k=0,1,\lyxmathsym{…},$ solves a consensus problem. More precisely,
let $z_{k}$’s be $n$ constants, then with the set of initial states
$z_{i}(0)=z_{i}$, the state of all agents asymptotically converges
to the average value $z=\frac{1}{n}\underset{i}{\sum}z_{i}$ provided
that the network is connected.

\subsection{Projection}

Let $z$ be an element in a Hilbert space $H$ and let $Z$ be a closed
(possibly non-convex) subset of $H$. We use $P_{Z}[\bar{z}]$ to
denote the projection of a vector $\bar{z}$ onto a closed convex
set $Z$, and define as follows:

\begin{eqnarray*}
P_{Z}[\bar{z}] & = & \underset{z\in Z}{\arg\min}\left\Vert \bar{z}-z\right\Vert 
\end{eqnarray*}

There is always at least one such point for each $z$, namely where
$H$ is a finite dimensional Hilbert space. If $Z$ is convex as well
as closed then each $z$ has exactly one projection point $P_{Z}[z]$
\citep{Luenberger1969}.

\subsection{Constrained Consensus Algorithm}

The constrained consensus problem is to achieve asymptotic consensus
on the local decision variables, $z_{i}$, through information exchange
with the neighboring nodes in the presence of the constraint sets,
$z_{i}\in Z_{i}$. A distributed algorithm for this problem was proposed
in \citep{Nedic2010}. In the algorithm, $i^{th}$ local variable
$z_{i}(k)$ is executed as follows: 

\begin{equation}
z_{i}(k+1)=P_{Z_{i}}[\overset{m}{\underset{j=1}{\sum}}a_{ij}(k+1)z_{j}(k)]\label{eq:con_consensus}
\end{equation}

and it can also be written:

\begin{gather}
v_{i}(k+1)=z_{i}(k)+\overset{m}{\underset{j=1}{\sum}}a_{ij}(k+1)(z_{j}(k)-z_{i}(k))\\
z_{i}(k+1)=P_{Z_{i}}[v_{i}(k+1)]
\end{gather}

Illustration of the algorithm in 2D case is given in Figure \ref{fig:Constrained-consensus-algorithm}.
The constrained consensus algorithm has several advantages over alternating
projection method, the most significant being that concurrent computation
of subspaces is possible. Accordingly, in a distributed setting, internal
dynamics of subspaces are coupled together, cooperating (or may be
competing) to achieve the overall system objectives:

\begin{gather}
v_{i}(k+1)=(1-a_{ij}(k+1))z_{i}(k)+a_{ij}(k+1)z_{j}(k)\\
z_{i}(k+1)=P_{Z_{i}}[v_{i}(k+1)]
\end{gather}

where $i\neq j,i,j=1,2$.

\begin{figure}
\begin{centering}
\includegraphics[scale=0.25]{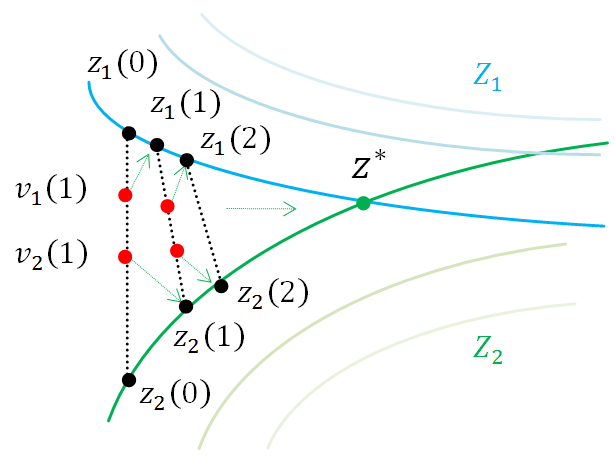}
\par\end{centering}

\caption{\label{fig:Constrained-consensus-algorithm}Constrained consensus
algorithm}
\end{figure}

\subsection{Constrained Optimization}

Distributed constrained optimization discusses the problem of optimizing
the sum of convex objective functions corresponding to $m$ connected
agents. The goal of the agents is to cooperatively solve the constrained
optimization problem 

\begin{align}
\min & \qquad\overset{m}{\underset{i=1}{\sum}}f_{i}\left(z\right)\label{eq:constrainedoptimization}\\
\text{subject to} & \qquad z\in\overset{m}{\underset{i=1}{\cap}}Z_{i}\nonumber 
\end{align}

where the local objective function of agent $i$, $f_{i}:\mathbb{R^{\mathrm{n}}\rightarrow\mathbb{R}}$,
is a convex function, and the local constraint set of agent $i$,
$Z_{i}\subseteq\mathbb{R^{\mathrm{n}}}$, is a closed convex set that
are known to agent only. In \citep{Nedic2010}, the following distributed
projected subgradient algorithm are proposed for solving problem (\ref{eq:constrainedoptimization})

\begin{align}
v_{i}\left(k\right)=\overset{m}{\underset{j=1}{\sum}} & a_{ij}\left(k\right)z_{j}\left(k\right)\label{eq:constrainedoptimizationalg0}\\
z_{i}\left(k+1\right)= & P_{Z_{i}}\left[v_{i}\left(k\right)-\alpha_{k}d_{i}\left(k\right)\right]\label{eq:constrainedoptimizationalg}
\end{align}

where the scalars $a_{ij}\left(k\right)$ are nonnegative weights
and the scalar $\alpha_{k}>0$ is a stepsize. The vector $d_{i}\left(k\right)$
is a subgradient of the agent local objective function $f_{i}\left(z\right)$
at $z=v_{i}\left(k\right)$.

\section{DISTRIBUTED MULTI-OBJECTIVE MULTI-DISCIPLINARY DESIGN OPTIMIZATION}

In this section we introduce the main problem we consider. We also
discuss the assumptions made in our model and propose our algorithm.
The general formulation for this section covering multidisciplinary
design optimization is in the spirit of the framework outlined in
design optimization literature \citep{Alexandrov1995}.

\subsection{MDO Problem Formulation }

The optimization problem usually encountered in many engineering system
design is considered to be a nonlinear programming problem. The general
multidisciplinary design optimization (MDO) problem is formulated
as follows: 

\begin{eqnarray}
 &  & \underset{z=\left(z^{l},z^{s}\right)}{\min}\quad f(x,y,z)\nonumber \\
 &  & \quad\text{s.t.}\quad\;\,\, g\left(x,y,z\right)\leq0\nonumber \\
 &  & \qquad\qquad h\left(x,y,z\right)=0\nonumber \\
 &  & \qquad\qquad\forall i,j\neq i,y_{i}=c_{ji}\left(x_{j},y_{j},z_{j}\right)\label{eq:generalMDO}\\
 &  & \qquad\qquad\forall i,x_{i}=T_{i}\left(x_{i},y_{i},z_{i}\right)\nonumber 
\end{eqnarray}

where 
\begin{verse}
$z=(z_{1},z_{2},\lyxmathsym{…},z_{m})$ denote design vector, consist
of design variables from different disciplines. Moreover, these variables
can be partitioned into shared and local design variables of ith subspace;
$z_{i}^{s}$ and $z_{i}^{l}$, respectively. 

$x=(x_{1},x_{2},\lyxmathsym{…},x_{m})$ is state vector and depend
on both linking and design variables. 

$y_{i}$’s denote linking variables and provide coupling among different
subsystems. 

$f,g$and $h$ are vector-valued objective function, inequality and
equality constraints functions, respectively. 

$c_{ji}:(x_{i}^{'},y_{i}^{'},z_{i}^{'})\rightarrow y_{j}^{''}$ denote
the coupling function from the subsystem $i$ to the subsystem $j$. 

$T_{i}:(x_{i}^{'},y_{i}^{'},z_{i}^{'})\rightarrow x_{i}^{''}$ denote
state transition functions that compute state variables of subspace
$i$. 
\end{verse}
We assume that the local objective function $f_{i}$ and the local
coupling $c_{ji}$ and transition $T_{i}$ functions are known to
agent only. Recall that $z^{*}$ is a local minimum of (\ref{eq:generalMDO})
if there exists $\varepsilon>0$ such that $f(x,y,z^{*})\leq f(x,y,z)$
for all $z\in S\bigcap B(z^{*},\varepsilon)$ and corresponding state
and coupling variables, where $S$ is the feasible region. 

The optimal design of complex engineering systems involves concurrent
optimization of several objectives, constrained by both local and
global issues. A distributed variant of equation (\ref{eq:generalMDO})
can be represented as follows: 

\begin{eqnarray}
 &  & \underset{z=\left(z^{i},z^{s}\right)}{\min}\quad\left[f_{0}(x,y,z);f_{1}(x,y,z);\ldots;f_{m}(x,y,z)\right]\label{eq:generalMMDO}\\
 &  & \quad\text{s.t.}\quad\;\,\, g_{i}\left(x,y,z\right)\leq0\qquad for\; all\; i\nonumber \\
 &  & \qquad\qquad h_{i}\left(x,y,z\right)=0\qquad for\; all\; i\nonumber \\
 &  & \qquad\qquad\forall i,j\neq i,y_{i}=c_{ji}\left(x_{j},y_{j},z_{j}\right)\nonumber \\
 &  & \qquad\qquad\forall i,x_{i}=T_{i}\left(x_{i},y_{i},z_{i}\right)\nonumber 
\end{eqnarray}

This problem can be formulated in a more compact form, which makes
it suitable for representation of distributed multi-objective design
optimization problems. Let $S$ denote the feasible region of (\ref{eq:generalMMDO}),
i.e., 

\begin{eqnarray*}
S_{i} & = & \left\{ (x^{'}\in X_{i},y^{'}\in Y_{i},z^{'}\in Z_{i}):\right.\\
 &  & \left.g_{i}(x^{'},y^{'},z^{'})\text{\ensuremath{\le}}0,h_{i}(x^{'},y^{'},z^{'})=0\right\} ,\, i=0,1,\text{…},m
\end{eqnarray*}

Accordingly, we use $\varphi_{i}$ to represent functional evolution
of the state and coupling variables of subspace $i$: 

\begin{eqnarray*}
\varphi_{i} & = & \left\{ (x_{i}\in X_{i},y_{i}\in Y_{i},z_{i}\in Z_{i}):x_{i}=T_{i}(x_{i},y_{i},z_{i}),\right.\\
 &  & \left.y_{i}=c_{ij}(x_{j},y_{j},z_{j}),i\text{\ensuremath{\neq}}j\right\} 
\end{eqnarray*}

and then let

\begin{eqnarray*}
S & = & \overset{m}{\underset{i=1}{\bigcap}}S_{i}
\end{eqnarray*}
\begin{eqnarray*}
\varphi & = & \overset{m}{\underset{i=1}{\cup}}\varphi_{i}
\end{eqnarray*}
. 

Then the problem (\ref{eq:generalMMDO}) can be written:

\begin{eqnarray}
 &  & \underset{z\in Z}{\min}F\left(x,y,z\right)=\left[f_{0}(x,y,z);f_{1}(x,y,z);\ldots;f_{m}(x,y,z)\right]\nonumber \\
 &  & \varphi\left(x,y,z\right)\in S
\end{eqnarray}

The goal of the agents is to cooperatively optimize both local and
global objectives: 

\begin{eqnarray*}
f(x) & = & f_{0}(x)+\underset{i=1}{\overset{m}{\sum}}f_{i}(x)
\end{eqnarray*}

\subsection{Distributed Computation Model }

In multiobjective multidisciplinary design problem, interconnections
and interactions among subspaces are usually complex. In such complex
problems, information- passing among subspaces and structural organization
of the subproblems is as important as subspaces optimization. In the
following, we discuss the proposed distributed computation model from
a graph theoretic perspective to handle information passing among
subproblems. Moreover, an analysis section is devoted to discuss basic
assumptions and corresponding results from algebraic graph theory,
which provides analytical framework, required for analysis of coordination
strategy and communication protocols. A detailed discussion of this
topic is beyond the scope of this paper and a brief overview of the
subject is given. For further information on this topic, readers are
referred to \citep{Nedic2010}. Recall we represent the structure
of our design optimization problem, consists of m subproblems by undirected
graph $G=(V,E)$. A typical structure of the design problem is illustrated
in Figure \ref{fig:Digraph-represent-structure}. 

\begin{figure}
\begin{centering}
\includegraphics[scale=0.25]{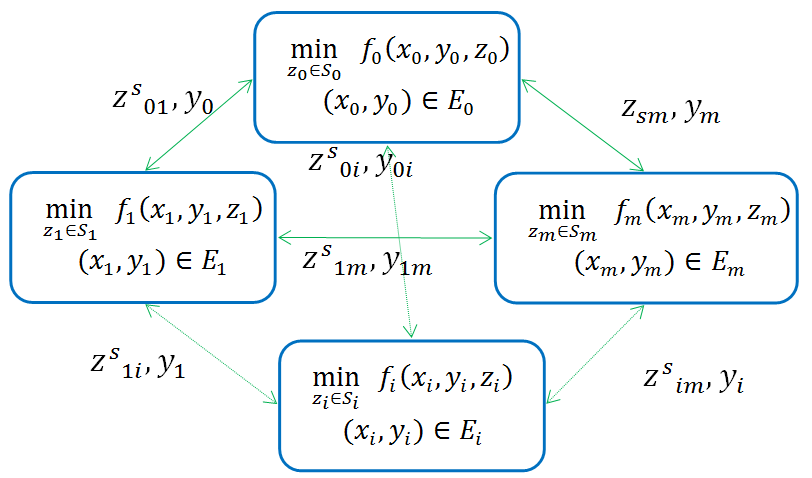}
\par\end{centering}

\caption{\label{fig:Digraph-represent-structure}Digraph represent structure
of design optimization problem}
\end{figure}

In MDO problems, there are at least two different types of information
being exchanged; shared design variables and coupled variables. Figure
\ref{fig:Flow-of-information;} illustrates a typical flow of information
among subspaces. 

\begin{figure}
\begin{centering}
\includegraphics[scale=0.2]{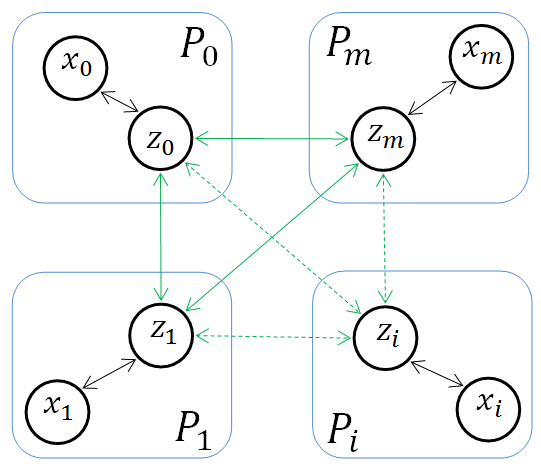} \enskip{}\includegraphics[scale=0.2]{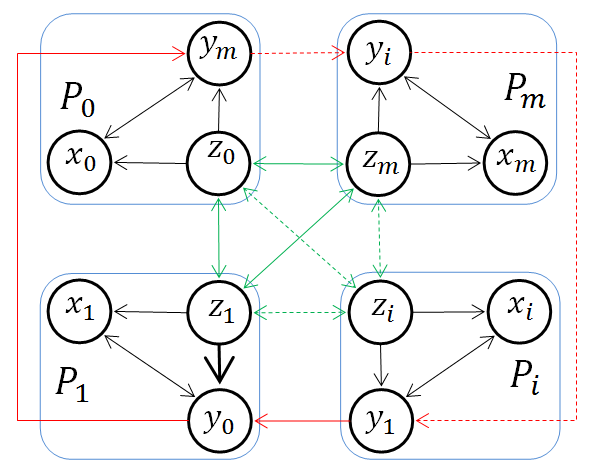}
\par\end{centering}

\caption{\label{fig:Flow-of-information;}A typical structure for information-passing;
(left) Distributed constrained optimization, (right) Distributed design
optimization; shared variables (green), coupling variables (red),
and state variables (black)}
\end{figure}

In addition, communication protocols play a critical role in providing
effective information sharing between several design teams. A communication
protocol consists of a set of rules which govern the orderly exchange
of information among entities. Two general classes of protocols are
proposed; synchronous and asynchronous protocols. Gossip communication
protocol of \citep{kempe2003} and the broadcast protocol of \citep{Aysal2009}
are some examples of these protocols. For a detailed review of this
topic see \citep{Aysal2009}.

\subsection{Collaborative Multidisciplinary Design Optimization (CMDO) algorithm }

We next consider the problem of general multi-objective multi-disciplinary
distributed design optimization (M2DO) (\ref{eq:generalMMDO}) corresponding
to m agents connected over a fixed topology. We also introduce two
different implementations of the proposed method for solving this
problem. Recall the local design variable associated with subspace
$i$ is denoted by $z_{i}(k)$ and is restricted to lie in a local
constraint set $Z_{i}$. In the multi-step version of the algorithm,
each agent $i$ starts with an initial estimate $z_{i}(0)\in Z_{i}$
and update its estimate by combining the estimates received from its
neighbors, by taking an optimization step to minimize its objective
function. Then she/he exchanges this estimate with neighbor subspaces.
Finally, the resulting value is projected into local feasible set.
This algorithm illustrates the intuition behind the proposed method
and allows employing heuristics in projection and optimization steps.
In interleaved version, the optimization step is augmented in the
constrained consensus algorithm. We consider this formulation when
discussing performance of the algorithm.

\begin{algorithm}
\begin{description}
\item [{\label{Algorithm-1:-Multi-step}Algorithm}] 1: Multi-step version
This algorithm is sequentially executed in three steps as follows: 
\item [{\textit{Update}}] Let’s consider the initial estimate of the local
and shared design variables $z_{i}^{l}(0)\in Z_{i}^{l},z_{i}^{s}(0)\in Z_{i}^{s}$,
respectively. Hereinafter, for the sake of notational simplicity,
we omit the super-index $s$ in $z_{i}^{s}$. These estimates $z_{i}$,
are exchanged directly among neighbor subspaces according to the problem
structure. Each agent updates its estimate by forming a convex combination
of the local estimate of shared design variable and that of neighbor
subspaces. This update mechanism at time $t_{k}$ is formally stated
by the following equation: 
\end{description}
\begin{eqnarray}
z_{i}(k+1) & = & z_{i}(k)+\alpha_{i}(k)\overset{m}{\underset{j=1}{\sum}}a_{ij}(k+1)(z_{j}(k)-z_{i}(k))
\end{eqnarray}

where $a_{ij}\left(k\right)$determines information passing among
subspaces. The coupling and state variables are also updated as follow: 

\begin{eqnarray}
i,j\neq i,y_{i} & = & c_{ji}(x_{j},y_{j},z_{j})\\
 &  & x_{i}=X_{i}(y_{i},z_{i})
\end{eqnarray}

At the end of this step, agents share their estimate of shared values
and coupling variables among each other.
\begin{description}
\item [{\textit{Optimization}}] Afterwards, each agent sets the updated
shared design variable, $z_{i}(k+1)$ as an initial value, updates
his/her objective based on the recent local state variables and received
coupling variables and then takes an optimization step according to
the following optimization problem $\underset{z_{i}}{\min}f_{i}(x_{i},y_{i},z_{i})$.
Let’s denote the resulting optimal design variables at time $t_{(k+1)}$
by $z_{i}^{*}(k+1)$. 
\item [{\textit{Projection}}] The last optimization step may result in
values outside the feasible region of each subspace. Therefore, we
project the resultant value into the feasible region to satisfy constraints
\begin{eqnarray}
z_{i}(k+1) & = & P_{S_{i}}[z_{i}^{*}(k+1)]
\end{eqnarray}

\end{description}
\caption{Multi-Step Algorithm}
\end{algorithm}

Although there are a lot of choices for projection or other steps,
but we do not discuss it in this paper. However, the main advantage
of this algorithm lies in the fact that it is possible to independently
specify each step. 

\begin{algorithm}
\begin{description}
\item [{Algorithm}] 2: In this algorithm, the optimization step is augmented
in the constrained consensus algorithm (\ref{eq:con_consensus}) as
follow
\end{description}
\begin{eqnarray}
 &  & z_{i}\left(k+1\right)=P_{S_{i}}\left[\alpha_{i}\left(k\right)\overset{m}{\underset{i=1}{\sum}}a_{ij}\left(k\right)z_{j}\left(k\right)-\gamma_{i}\left(k\right)d_{i}\left(k\right)\right]\label{eq:alg2}
\end{eqnarray}

Where the vector $d_{i}(k+1)$ is an optimization step of the objective
function of subspace $i$, i.e., $f_{i}$ at point $(x_{i}(k),y_{i}(k),z_{i}(k))$,
updated at each time step $t_{k}$.

\caption{\label{alg:Interleaved-Algorithm}Interleaved Algorithm}
\end{algorithm}

An illustration of the algorithm is presented in Figure \ref{fig:Illustration-of-the}.
In the following sections, we discuss the behavior of this algorithm. 

\begin{figure}
\begin{centering}
\includegraphics[scale=0.25]{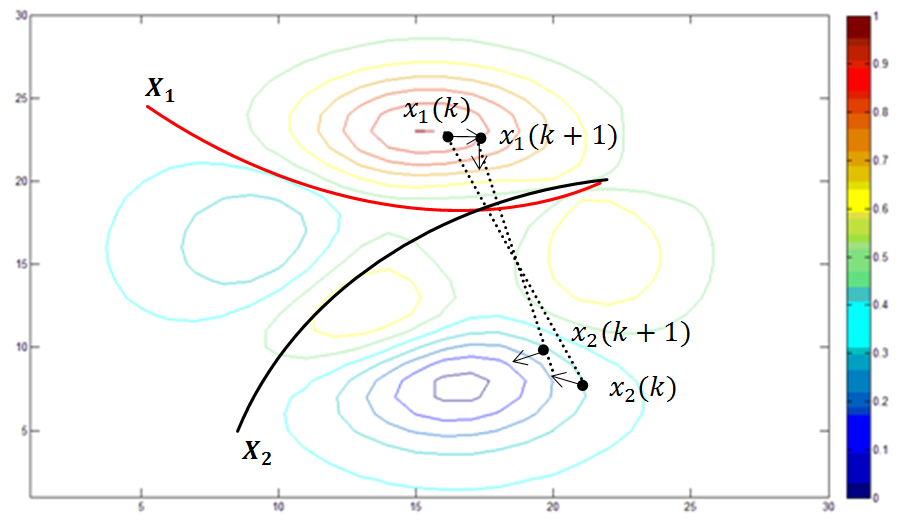}
\par\end{centering}

\caption{\label{fig:Illustration-of-the}Illustration of the proposed algorithm}
\end{figure}

\section{PERFORMANCE ANALYSIS}

We next discuss the behavior of the proposed algorithm. Analysis of
the general MDO problem is difficult because of the coupling among
subspaces and internal state transition. In this paper, we focus on
some special cases including quadratic objective functions with linear
subspaces. Firstly, some existing results in the field are discussed.
Then, a class of multidisciplinary design optimization problems are
defined and some extra assumptions are placed on information structure
and coordination strategy that can guarantee the convergence of the
algorithm.

\subsection{Constrained Consensus }

This section provides a summary description of the existing convergence
results for distributed optimization. We adopt the following assumptions
in our analysis following Blondel et al. \citep{Blondel2005} and
\citep{Nedic2010}. 

Assumption1: (\textit{Network Topology}) We assume information passing
among subspaces takes place at discrete time steps, $t_{k}$. Recall
the weight associated with the edges in network graph can be denoted
by $a_{ij}(k)$, or $[A(k)]_{ij}$, where $A(k)$ is called the \textit{network
weight matrix}. 

Assumption2: (\textit{Weights Rule}) There exists a scalar $\eta$
with $0<\eta<1$ such that for all $i,j\in\left\{ 1,...,m\right\} $,
and $k\ge0$
\begin{enumerate}
\item $a_{ii}(k)\ge\eta$
\item $a_{ij}(k)\ge\eta$ when subspace j communicates with subspace $i$,
and $a_{ij}(k)=0$ otherwise. 
\item $\overset{m}{\underset{i=1}{\sum}}a_{ij}(k)=1$(\textit{row stochastic}) 
\end{enumerate}
Assumption3: (\textit{Connectivity and symmetry}) The graph $(V,E_{\infty})$
is connected, where $E_{\infty}$ is the set of edges $(j,i)$ representing
subspace pairs communicating directly infinitely many times, i.e.,
\begin{eqnarray*}
E_{\infty} & = & \left\{ (j,i)|(j,i)\in E_{k}\textrm{ for infinitely many indices }k\right\} 
\end{eqnarray*}

Moreover, the graph $(V,E_{\infty})$ is symmetric, i.e., the weights
satisfy $a_{ij}(k)=a_{ji}(k)$ for all $i,j$. 

Assumption4: (\textit{Bounded intercommunication intervals}) If $i$
communicates to $j$ an infinite number of times {[}that is, if $(i,j)\in E(t)$
infinitely often{]}, then there is some $B$ such that, for all $t$,

\begin{eqnarray*}
(i,j) & \in & E(t)\cup E(t+1)\cup\cdots\cup E(t+B-1)
\end{eqnarray*}
. 

Assumption5:(\textit{Doubly Stochasticity}) The vectors 

$a_{i}(k)=(a_{i1}(k),\lyxmathsym{…},a_{im}(k))^{'}$ satisfy: 

a) $a_{i}(k)\ge0$ and $\overset{m}{\underset{i=1}{\sum}}a_{ij}(k)=1$
for all $i$ and $k$, i.e., the vector $a_{i}(k)$. 

b) $\overset{m}{\underset{i=1}{\sum}}a_{ij}(k)=1$ for all $j$ and
$k$. 
\begin{remark}
This assumption establishes that each agent takes a convex combination
of its estimate and the estimates of its neighbors. Moreover, second
part of this assumption together with Assumption 2, ensures that the
estimate of every agent is influenced by the estimates of every other
agent with the same frequency in the limit, i.e., all agents are equally
influential in the long run \citep{Nedic2010}. \end{remark}
\begin{lemma}
\citep{olfati2007} Let $G$ be a connected undirected graph. Then,
the algorithm (\ref{eq:avg_con}) asymptotically solves an average
consensus problem for all initial states. \end{lemma}
\begin{remark}
Simply, this lemma states that if there exist some paths for flow
of information, consensus is eventually achieved. Therefore, information
exchange among different subspaces with the same shared design variables
are necessary for convergence. \end{remark}
\begin{proposition}
\label{prop:pro-conscon}\citep{Nedic2010} (\textit{Consensus}) Let
the set $Z=\overset{m}{\underset{i=1}{\cap}}Z_{i}$ be nonempty. Also,
let Weights Rule, Doubly Stochasticity, Connectivity, and Information
Exchange Assumptions hold (cf. \textit{Assumptions} 2, 3, 4, and 5).
For all $i$, let the sequence $\left\{ z_{i}(k)\right\} $ be generated
by the constrained consensus algorithm (\ref{eq:con_consensus}).
We then have for some $\tilde{z}\in Z$ and all $i$,$\underset{k\rightarrow\infty}{\lim}\left\Vert z_{i}\left(k\right)-\tilde{z}\right\Vert =0$. 
\end{proposition}

\subsection{Distributed Constrained Optimization}

In this section, we discuss the existing results for the constrained
optimization problem. Let's first consider the following assumptions.

Assumption 6: (\textit{Same Constraint Set})
\begin{itemize}
\item The constraint sets $Z_{i}$ are the same, i.e, $Z_{i}=Z$ for a closed
convex set $Z$. 
\item The subgradient sets of each $f_{i}$ are bounded over the set $Z$,
i.e., there is a scalar $L>0$ such that for all $i$
\begin{align*}
\left\Vert d\right\Vert \leq L & \text{ for all }d\in\partial f_{i}\left(z\right)\text{ and all }z\in Z
\end{align*}

\end{itemize}
Assumption 7: (\textit{Compactness}) For each $i$, the local constraint
set $Z_{i}$ is a compact set, i.e., there exists a scalar $B>0$
such that

\begin{align*}
\left\Vert z\right\Vert \leq B & \text{ for all }z\in Z_{i}\text{ and all }i
\end{align*}

The next proposition presents convergence result for the same constraint
set case. In particular, it is shown that the iterates of the projected
subgradient algorithm (\ref{eq:constrainedoptimizationalg0})-(\ref{eq:constrainedoptimizationalg})
converge to an optimal solution when we use a stepsize converging
to zero fast enough.
\begin{proposition}
\label{prop:SameConPro}\citep{Nedic2010} Let Weights Rule, Doubly
Stochasticity, Connectivity, Information Exchange, and Same Constraint
Set Assumptions hold (cf. Assumptions 2, 3, 4, 5, and 6). Let $\left\{ z_{i}(k)\right\} $
be the iterates generated by the algorithm (\ref{eq:constrainedoptimizationalg0})-(\ref{eq:constrainedoptimizationalg})
with the stepsize satisfying $\sum_{k}\alpha_{k}=\infty$ and $\sum_{k}\alpha_{k}^{2}<\infty$.
In addition, assume that the optimal solution set $Z^{*}$is nonempty.
Then, there exists an optimal point $z^{*}\in Z^{*}$ such that

\begin{eqnarray*}
\underset{k\rightarrow\infty}{\lim} & \left\Vert z_{i}\left(k\right)-z^{*}\right\Vert  & =0\text{ for all }i
\end{eqnarray*}

\end{proposition}
The next proposition presents convergence result for the projected
subgradient algorithm (\ref{eq:constrainedoptimizationalg0})-(\ref{eq:constrainedoptimizationalg})
in the uniform weight case.
\begin{proposition}
\label{prop:UniWeiPro}\citep{Nedic2010} Let Interior Point and Compactness
Assumptions hold (cf. Assumptions 1 and 7). Let $\left\{ z_{i}(k)\right\} $
be the iterates generated by the algorithm (\ref{eq:constrainedoptimizationalg0})-(\ref{eq:constrainedoptimizationalg})
with the weight vectors $a_{i}\left(k\right)=\left(1/m,\ldots,1/m\right)^{'}$
for all $i$ and $k$, and the stepsize satisfying $\sum_{k}\alpha_{k}=\infty$
and $\sum_{k}\alpha_{k}^{2}<\infty$. Then, the sequences $\left\{ z_{i}(k)\right\} $,
$i=1,\ldots,m$, converge to the same optimal point, i.e.

\begin{align*}
\underset{k\rightarrow\infty}{\lim} & \left\Vert z_{i}\left(k\right)-z^{*}\right\Vert =0\text{ for some }x^{*}\in X^{*}\text{ and all }i
\end{align*}

\end{proposition}

\subsection{Collaborative Multidisciplinary Design Optimization}

In this section, we present our results. We consider a distributed
multidisciplinary design optimization problem with quadratic objectives
and linear constraints. In a special case, we establish conditions
under which convergence of the algorithm (\ref{eq:alg2}) to optimial
solutions is guaranteed. For notational convenience, let $w$ denote 

\begin{eqnarray*}
w & = & \left[\begin{array}{c}
w_{x}\\
w_{y}\\
w_{z}
\end{array}\right]=\left[\begin{array}{c}
x\\
y\\
z
\end{array}\right]
\end{eqnarray*}

\begin{definition}
The design optimization problem of optimizing a quadratic function
of several variables, including (shared) design variables as well
as state and coupling variables, subject to linear constraints on
these variables is called Quadratic Design Programming (QDP). QDP
is a special type of multidisciplinary design optimization problem
and can be formulated as follow:
\end{definition}
\begin{eqnarray}
 &  & \underset{w_{z}}{\min}\quad w^{T}Qw+P^{T}w\nonumber \\
 &  & \quad s.t.\quad\;\, Aw\leq b\nonumber \\
 &  & \qquad\qquad Ew=d\label{eq:QDP}\\
 &  & \qquad\qquad y^{'}=Cw\nonumber \\
 &  & \qquad\qquad x^{'}=Xw\nonumber 
\end{eqnarray}

\begin{definition}
An special class of QDP problem can be defined as follow:

\begin{eqnarray}
 &  & \underset{z}{\min}\quad\overset{m}{\underset{i=1}{\sum}}c_{i}\left(k\right)\left(z_{i}\left(k\right)-y_{i}\left(k\right)-t_{i}\left(k\right)\right)^{2}\label{eq:SpecialQDP}\\
 &  & \qquad\qquad\overset{m}{\underset{r=1}{\sum}}\left(\lambda_{r}z_{r}\left(k\right)+\mu_{r}y_{r}\left(k\right)\right)\leq b\nonumber \\
 &  & \quad\text{s.t.}\quad\;\, y_{i}\left(k\right)=\overset{m}{\underset{r=1}{\sum}}b_{ir}\left(k\right)z_{r}\left(k\right),\quad i=1.\ldots,m\nonumber 
\end{eqnarray}

\end{definition}
The next proposition presents convergence result for algorithm (\ref{eq:alg2})
for design problem \ref{eq:SpecialQDP}. 
\begin{proposition}
\label{prop:MyPropos1} Let the sequence $\left\{ z_{i}(k)\right\} $
be generated by the collaborative design optimization algorithm (\ref{eq:alg2})
for the QDP problem (\ref{eq:SpecialQDP}) with the following weights,
for all $i$ and $j$

\begin{eqnarray*}
a_{ij}^{'}\left(k\right) & = & a_{ij}\left(k\right)-2\alpha_{i}(k)c_{j}\left(k\right)b_{j}^{'}\left(k\right)^{2}\\
 & b_{i}^{'} & \left(k\right)=1-\overset{m}{\underset{r=1}{\sum}}b_{ri}\left(k\right)
\end{eqnarray*}

Also, let Weights Rule, Doubly Stochasticity, Connectivity, and Information
Exchange Assumptions hold (cf. \textit{Assumptions} 2, 3, 4, and 5)
and, let the set $Z=\overset{m}{\underset{i=1}{\cap}}Z_{i}$ be nonempty.
Under Same Constraint Set Assumptions with the stepsize satisfying
$\sum_{k}\alpha_{k}=\infty$ and $\sum_{k}\alpha_{k}^{2}<\infty$.
Then, there exists an optimal point $z^{*}\in Z^{*}$ such that
\end{proposition}
\begin{align*}
\underset{k\rightarrow\infty}{\lim} & \left\Vert z_{i}\left(k\right)-z^{*}\right\Vert =0\text{ for all }i
\end{align*}

\begin{proof}
The proof idea is straightforward. It is to show that the collaborative
design optimization algorithm (\ref{eq:alg2}) for QDP problem reduces
to the distributed constrained optimization equation (\ref{eq:con_consensus}).
First, we calcualte the gradient of quadratic objective function:

\begin{eqnarray*}
d_{i}^{'}\left(k\right) & = & \nabla_{z}f\left(y,z\right)\\
 & = & \nabla_{z}\left[\overset{m}{\underset{i=1}{\sum}}c_{i}\left(k\right)\left(z_{i}\left(k\right)-y_{i}\left(k\right)-t_{i}\left(k\right)\right)^{2}\right]\\
 & = & \nabla_{z}\left[\overset{m}{\underset{i=1}{\sum}}c_{i}\left(k\right)\left(b_{i}^{'}\left(k\right)z_{i}\left(k\right)-t_{i}\left(k\right)\right)^{2}\right]\\
 & = & \nabla_{z}\left[\overset{m}{\underset{i=1}{\sum}}c_{i}\left(k\right)\left(b_{i}^{'}\left(k\right)z_{i}\left(k\right)-t_{i}\left(k\right)\right)^{2}\right]\\
 & = & \overset{m}{\underset{i=1}{\sum}}c_{i}\left(k\right)\nabla_{z}\left(b_{i}^{'}\left(k\right)z_{i}\left(k\right)-t_{i}\left(k\right)\right)^{2}\\
 & = & \overset{m}{\underset{i=1}{\sum}}2c_{i}\left(k\right)b_{i}^{'}\left(k\right)\left(b_{i}^{'}\left(k\right)z_{i}\left(k\right)-t_{i}\left(k\right)\right)
\end{eqnarray*}

where

\begin{align*}
b_{i}^{'}\left(k\right) & =1-\overset{m}{\underset{r=1}{\sum}}b_{ri}\left(k\right)
\end{align*}

By substituting in equation (\ref{eq:alg2})

\begin{multline*}
z_{i}\left(k+1\right)=P_{S_{i}}\left[\overset{m}{\underset{j=1}{\sum}}a_{ij}\left(k\right)z_{j}\left(k\right)-\alpha_{i}\left(k\right)d_{i}^{'}\left(k\right)\right]\\
=P_{S_{i}}\left[\overset{m}{\underset{j=1}{\sum}}a_{ij}\left(k\right)z_{j}\left(k\right)\right.\cdots\\
\cdots\left.-\alpha_{i}\left(k\right)\overset{m}{\underset{i=1}{\sum}}2c_{i}\left(k\right)b_{i}^{'}\left(k\right)\left(b_{i}^{'}\left(k\right)z_{i}\left(k\right)-t_{i}\left(k\right)\right)\right]\\
=P_{S_{i}}\left[\overset{m}{\underset{j=1}{\sum}}\left(\left(a_{ij}\left(k\right)-2\alpha_{i}(k)c_{j}\left(k\right)b_{j}^{'}\left(k\right)^{2}\right)z_{j}(k)\right.\cdots\right.\\
\left.\cdots\left.-\left(2\alpha_{i}(k)c_{j}\left(k\right)b_{j}^{'}\left(k\right)^{2}\right)t_{j}\left(k\right)\right)\right]
\end{multline*}

we conclude 

\begin{align*}
z_{i}\left(k+1\right) & =P_{S_{i}}\left[\overset{m}{\underset{j=1}{\sum}}\left(a_{ij}^{'}\left(k\right)z_{j}\left(k\right)-\beta_{ij}\left(k\right)t_{j}\left(k\right)\right)\right]
\end{align*}

where

\begin{align*}
a_{ij}^{'}\left(k\right) & =a_{ij}\left(k\right)-2\alpha_{i}(k)c_{j}\left(k\right)b_{j}^{'}\left(k\right)^{2}\\
 & \beta_{ij}\left(k\right)=2\alpha_{i}(k)c_{j}\left(k\right)b_{j}^{'}\left(k\right)^{2}
\end{align*}

Afterwards, the results are concluded by using proposition (\ref{prop:SameConPro}).\end{proof}
\begin{remark}
Every distributed design optimization problem without design variables
updates (i.e., $\alpha_{i}(k)=0$) reduces to a distributed average
consensus problem. 
\end{remark}

\begin{remark}
It is possible to guarantee convergence of a distributed design optimization
problem both through sharing appropriate design variables (i.e., $a_{ij}\left(k\right)$)
as well as coupling variables (i.e., $b_{j}^{'}\left(k\right)$).\end{remark}
\begin{proposition}
\label{prop:MyProPos2}Let the sequence $\left\{ z_{i}(k)\right\} $
be generated by the collaborative design optimization algorithm (\ref{eq:alg2})
for the QDP problem (\ref{eq:SpecialQDP}) with the following weights,
for all $i$ and $j$

\begin{align*}
a_{ij}^{'}\left(k\right) & =a_{ij}\left(k\right)-2\alpha_{i}(k)c_{j}\left(k\right)b_{j}^{'}\left(k\right)^{2}\\
 & b_{i}^{'}\left(k\right)=1-\overset{m}{\underset{r=1}{\sum}}b_{ri}\left(k\right)
\end{align*}

Also, let Weights Rule, Doubly Stochasticity, Connectivity, and Information
Exchange Assumptions hold (cf. \textit{Assumptions} 2, 3, 4, and 5)
and, let the set $Z=\overset{m}{\underset{i=1}{\cap}}Z_{i}$ be nonempty.
Let the weight vectors 
\begin{align*}
a_{i}\left(k\right) & =\left(1/m,\ldots,1/m\right)^{'}\text{ for all \ensuremath{i}and \ensuremath{k}}
\end{align*}
and the stepsize satisfying $\sum_{k}\alpha_{k}=\infty$ and $\sum_{k}\alpha_{k}^{2}<\infty$.
Then, the sequences $\left\{ z_{i}(k)\right\} $, $i=1,\ldots,m$,
converge to the same optimal point, i.e.

\begin{align*}
\underset{k\rightarrow\infty}{\lim} & \left\Vert z_{i}\left(k\right)-z^{*}\right\Vert =0\text{ for some }z^{*}\in Z^{*}\text{ and all }i
\end{align*}
\end{proposition}
\begin{proof}
Same as proposition (\ref{prop:MyPropos1}).\end{proof}
\begin{proposition}
\label{prop:MyProPos3}Let the sequence $\left\{ z_{i}(k)\right\} $
be generated by the collaborative design optimization algorithm (\ref{eq:alg2})
for the QDP problem (\ref{eq:SpecialQDP}) with $t_{i}\left(k\right)=0$
and the following weights, for all $i$ and $j$

\begin{align*}
a_{ij}^{'}\left(k\right) & =a_{ij}\left(k\right)-2\alpha_{i}(k)c_{j}\left(k\right)b_{j}^{'}\left(k\right)^{2}\\
 & b_{i}^{'}\left(k\right)=1-\overset{m}{\underset{r=1}{\sum}}b_{ri}\left(k\right)
\end{align*}

Also, let Weights Rule, Doubly Stochasticity, Connectivity, and Information
Exchange Assumptions hold (cf. \textit{Assumptions} 2, 3, 4, and 5)
and, let the set $Z=\overset{m}{\underset{i=1}{\cap}}Z_{i}$ be nonempty.
We then have for some $\tilde{z}\in Z$ and all $i$,\\
\begin{eqnarray*}
\underset{k\rightarrow\infty}{\lim}\left\Vert z_{i}\left(k\right)-\tilde{z}\right\Vert  & = & 0
\end{eqnarray*}
\end{proposition}
\begin{proof}
Same as proposition (\ref{prop:MyPropos1}). The proof idea is to
show that the collaborative design optimization algorithm (\ref{eq:alg2})
for QDP problem reduces to the projected consensus algorithm (\ref{eq:constrainedoptimizationalg0})-(\ref{eq:constrainedoptimizationalg})
when $t_{i}\left(k\right)=0$. \end{proof}
\begin{remark}
Discussion of another proof of propositions (\ref{prop:MyPropos1}),
(\ref{prop:MyProPos2}), and (\ref{prop:MyProPos3}), provided by
expanding the projection onto hyperplanes, omitted here due to space
limitations (can be found in \citep{Noori2012}). 
\end{remark}
Another special case that may be of interest is Quasi-seperable problems,
(i.e., the problems where coupling objectives and constraints are
not present). Further discussions of the different special cases will
be omitted here in order to reserve more space and time for the discussion
of the general problem. The interested reader is referred to \citep{Noori2012}.

\section{PRELIMINARY NUMERICAL RESULTS}

The coupling among subspaces and state transitions are the main difference
between distributed optimization and design optimization problems
(see Figure \ref{fig:Flow-of-information;}). In fact, guaranteed
convergence of the general distributed design optimization problem
in the presence of coupling dynamics is difficult. In special cases,
it can be shown that this problem can be reduced to the distributed
constrained optimization problem. In the following example, we consider
an special class of QDP problem defined in (\ref{eq:SpecialQDP}).
\begin{example}
Let's consider the following design optimization problem P which is
composed of four subproblems $P_{1}$,$P_{2}$,$P_{3}$ and $P_{4}$.
The objectives contains both local and shared design variables, linked
through coupling variables. For compliance with the definition of
QDP problem, subproblems are defined according to (\ref{eq:QDP}).
The structure of the problem is illustrated in Figure \ref{fig:Example-1}.

\begin{figure}
\begin{centering}
\includegraphics[scale=0.28]{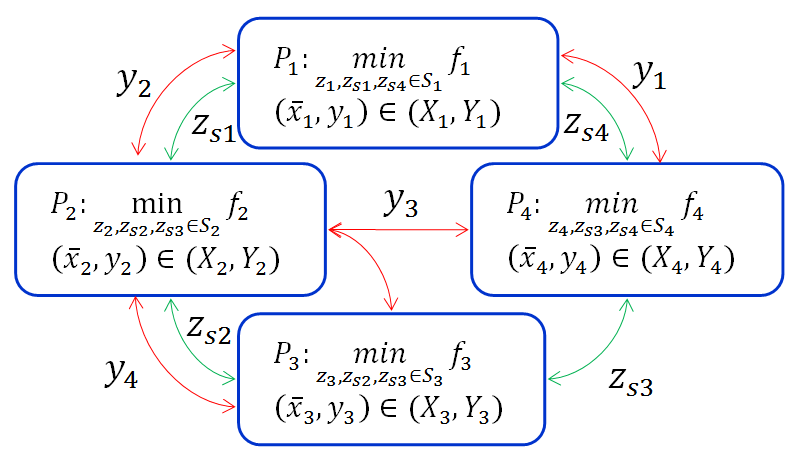}
\par\end{centering}

\caption{\label{fig:Example-1}Example 1}
\end{figure}

\begin{eqnarray*}
P_{1}: &  & \underset{z_{1},z_{s1},z_{s4}}{\min}\quad f_{1}=\left(z_{1}\left(t\right)-y_{2}\left(t\right)\right)^{2}+\left(z_{s1}\left(t\right)-10\right)^{2}\\
 &  & \qquad\qquad\qquad+\left(z_{s4}\left(t\right)-10\right)^{2}\\
 &  & \quad s.t.\quad\;\,\, z_{1}\left(t\right)+z_{s1}\left(t\right)-z_{s4}\left(t\right)\leq1\\
 &  & \qquad\qquad x_{1}\left(t\right)-z_{3}\left(t\right)-z_{4}\left(t\right)=0\\
 &  & \qquad\qquad y_{1}\left(t\right)=x_{1}\left(t\right)
\end{eqnarray*}

\begin{eqnarray*}
P_{2}: &  & \underset{z_{2},z_{s1},z_{s2}}{\min}\quad f_{2}=\left(z_{2}\left(t\right)+y_{1}\left(t\right)\right)^{2}+\left(z_{s1}\left(t\right)-4\right)^{2}\\
 &  & \qquad\qquad\qquad+\left(z_{s2}\left(t\right)-4\right)^{2}\\
 &  & \quad s.t.\quad\;\, z_{2}\left(t\right)-z_{s1}\left(t\right)-z_{s2}\left(t\right)\leq1\\
 &  & \qquad\qquad x_{2}\left(t\right)-z_{2}\left(t\right)=0\\
 &  & \qquad\qquad y_{2}\left(t\right)=x_{2}\left(t\right)
\end{eqnarray*}

\begin{eqnarray*}
P_{3}: &  & \underset{z_{3},z_{s2},z_{s3}}{\min}\quad f_{3}=\left(z_{3}\left(t\right)-y_{4}\left(t\right)\right)^{2}+\left(z_{s2}\left(t\right)+2\right)^{2}\\
 &  & \qquad\qquad\qquad+\left(z_{s3}\left(t\right)-5\right)^{2}\\
 &  & \quad s.t.\quad\;\,\, z_{3}\left(t\right)+z_{s2}\left(t\right)+z_{s3}\left(t\right)\geq-1\\
 &  & \qquad\qquad x_{3}\left(t\right)-z_{s2}\left(t\right)=0\\
 &  & \qquad\qquad y_{3}\left(t\right)=x_{3}\left(t\right)
\end{eqnarray*}

\begin{eqnarray*}
P_{4}: &  & \underset{z_{4},z_{s3},z_{s4}}{\min}\quad f_{4}=\left(z_{4}\left(t\right)+y_{3}\left(t\right)\right)^{2}+\left(z_{s3}\left(t\right)+10\right)^{2}\\
 &  & \qquad\qquad\qquad+\left(z_{s4}\left(t\right)+10\right)^{2}\\
 &  & \quad s.t.\quad\;\,\, z_{4}\left(t\right)-z_{s3}\left(t\right)+z_{s4}\left(t\right)\geq-1\\
 &  & \qquad\qquad x_{4}\left(t\right)-z_{s1}\left(t\right)=0\\
 &  & \qquad\qquad y_{4}\left(t\right)=x_{4}\left(t\right)
\end{eqnarray*}

Each quadratic objective function are constrained on a one-dimensional
locus (i.e., a line) and within a half-space with a hyperplane. The
early implementation of the algorithm is done on MATLAB$^{\lyxmathsym{\textregistered}}$
\citep{MathWorks2010}, in a 32-bit environment. The testbed environment
consists of 4 Intel Pentium IV. 2,5GHz and 1GB RAM workstations running
Microsoft Windows XP Professional. The PCs are interconnected by a
100Mbit Ethernet LAN setted-up as a single collision domain. 

The test results are presented in table \ref{tab:Results}. The figure
\ref{fig:Results} shows the performance of the algorithm that uses
the following settings; optimization coefficient$\alpha_{opt}=\frac{0.1}{m}$,
consensus coefficient$\alpha_{con}=\frac{0.1}{m}$, and iteration
$n=10000$. As depicted in Figure \ref{fig:Results}, while convergence
to the optimal solution is quickly achieved (about 100 iterations)
but approaching to accurate result is slow (10000 iterations). This
feature of the algorithm could be improved in several ways; amongst
them are proper network weight design or proper selection of optimization
and consensus factors. 

\begin{table}
\begin{centering}
\caption{\label{tab:Results}Comparision of simulation results}

\par\end{centering}

\centering{}%
\begin{tabular}{ccc}
\hline 
Var. & AIO method & CMDO method\tabularnewline
\hline 
\hline 
$z_{1}$ & -6.000 & -5.999\tabularnewline
$z_{2}$ & -6.000 & -6.000\tabularnewline
$z_{3}$ & 7.000 & 7.000\tabularnewline
$z_{4}$ & -0.999 & -1.000\tabularnewline
$z_{s1}$ & 7.000 & 7.000\tabularnewline
$z_{s2}$ & 0.999 & 1.000\tabularnewline
$z_{s3}$ & -2.499 & -2.499\tabularnewline
$z_{s4}$ & 0.000 & 0.000\tabularnewline
\hline 
\end{tabular}
\end{table}

\begin{figure}
\begin{centering}
\includegraphics[scale=0.22]{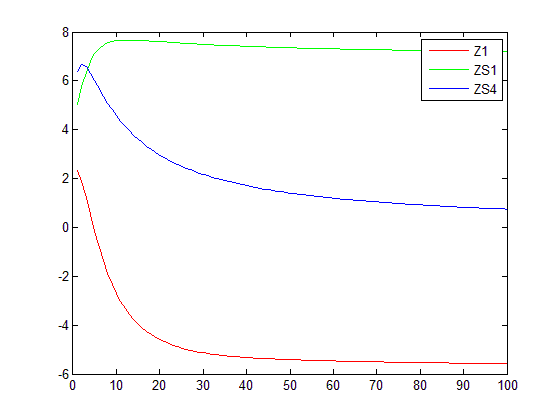}\includegraphics[scale=0.22]{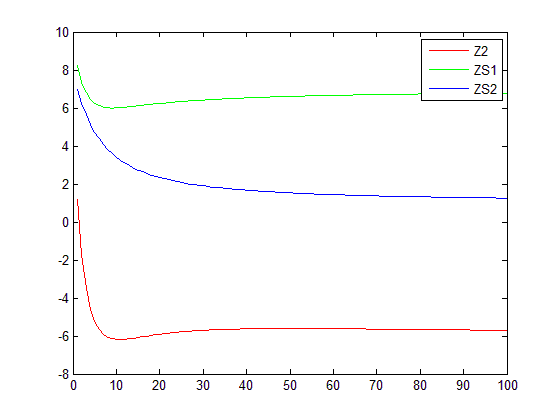}
\par\end{centering}

\begin{centering}
\includegraphics[scale=0.22]{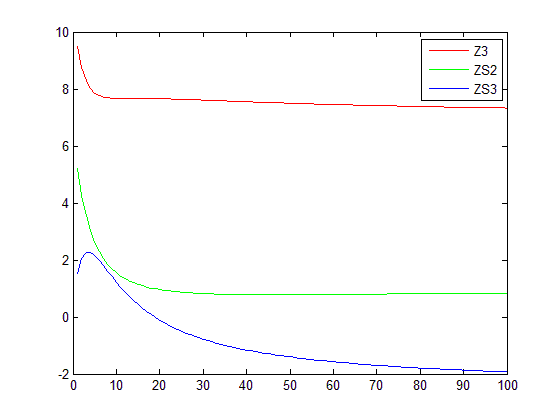}\includegraphics[scale=0.22]{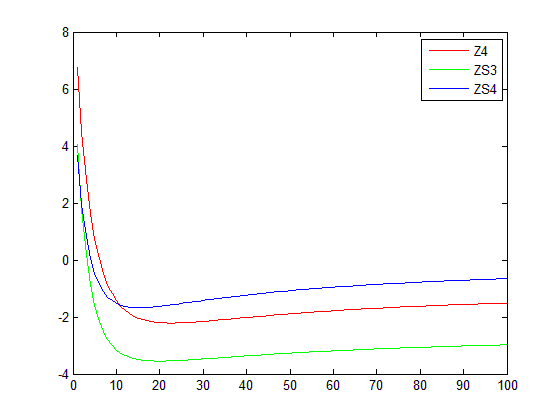}
\par\end{centering}

\caption{\label{fig:Results}Convergence of the solutions }

\end{figure}

\end{example}

\section{CONCLUSION }

In this paper, a distributed computation framework for multi-objective
multidisciplinary design optimization problems is proposed. The corresponding
coordination strategy is collaborative and concurrent, which make
it suitable for real-world design problems. It is also shown that
distributed constrained optimization problem is an special case of
collaborative multidisciplinary design optimization problem. We also
investigate an important class of design optimization problems, called
QDP problem. By using existing results, we established convergence
of an special case of QDP problem. Finally, the paper highlights challenging
areas in which research is required to allow us to utilize the full
potential of distributed optimization methods in multidisciplinary
design optimization in the future. 

\bibliographystyle{plainnat}
\bibliography{MyBibDataBase}

\end{document}